\documentclass[12pt,letterpaper,reqno,toc=flat]{amsart}
\usepackage{amsmath,amssymb,amsthm,latexsym}
\usepackage[margin=1.0in]{geometry}
\usepackage{graphicx}
\usepackage{tkz-graph}
\usepackage{float}

\def\cal{\mathcal}
\newtheorem{theorem}{Theorem}[section]
\newtheorem{lemma}[theorem]{Lemma}
\newtheorem{sublemma}[theorem]{Sublemma}

\newtheorem{definition}[theorem]{Definition}

\newtheorem{remark}[theorem]{Remark}

\begin{document}

\title{A scrapbook of inadmissible line complexes for the X-ray transform}

\author{ Eric Grinberg \& Mehmet Orhon}
\address{UMass Boston \&  University of New Hampshire}
\email{eric.grinberg@umb.edu and mo@unh.edu}

\subjclass[2010]{44A12} \keywords{Radon transform, admissible complexes, finite fields, integral geometry}

\begin{abstract}
We consider a finite field model of the X-ray transform that integrates functions along lines in dimension 3, within the context  of finite fields. The admissibility problem asks for minimal sets of lines for which the restricted transform is invertible. Graph theoretic conditions are known which characterize admissible collections of lines, and these have been counted using a brute force computer program. Here we perform the count by hand and, at the same time, produce a detailed illustration of the possible structures of inadmissible complexes. The resulting scrapbook may be of interest in an artificial intelligence approach to enumerating and illustrating admissible complexes in arbitrary dimensions (arbitrarily large ambient spaces, with transforms integrating over subspaces  of arbitrary dimensions.)
\end{abstract}

\maketitle

\parskip=0.36em
\tableofcontents
\parskip=5mm 

\medskip

\section{Telegraphic Introduction: $937$,$438$ or  $937$,$440$ ?}

This paper is a continuation of  \cite{Gr2}, which provides motivation and background. The X-ray transform, or Radon transform \cite{Radon1917},  that integrates functions along lines in $\mathbb R^3$, is  at the mathematical heart of  CAT scanners. Image reconstruction employs the principle that a well behaved function is determined by its line integrals.  Dimension counting shows that not all lines are needed. It is interesting to consider minimal families of lines that enable  reconstruction. In the continuous category  some analytic or topological restrictions are imposed on the notion of ``minimal family". We can take a discrete, or even finite model, and there, no restrictions are needed. In principle, ``all" questions may be answered concretely in such a context. The smallest model replaces the real numbers  $\mathbb R$ by the  two element  field $\mathbb F_2$, and ``space" and lines within it  can then be illustrated as in figure $1$  below.   This theme is explored in a number of papers, e.g., \cite{BGK,Gr1,Gr2}.

\begin{figure}[h!]
\centerline{
\includegraphics[scale=0.4]{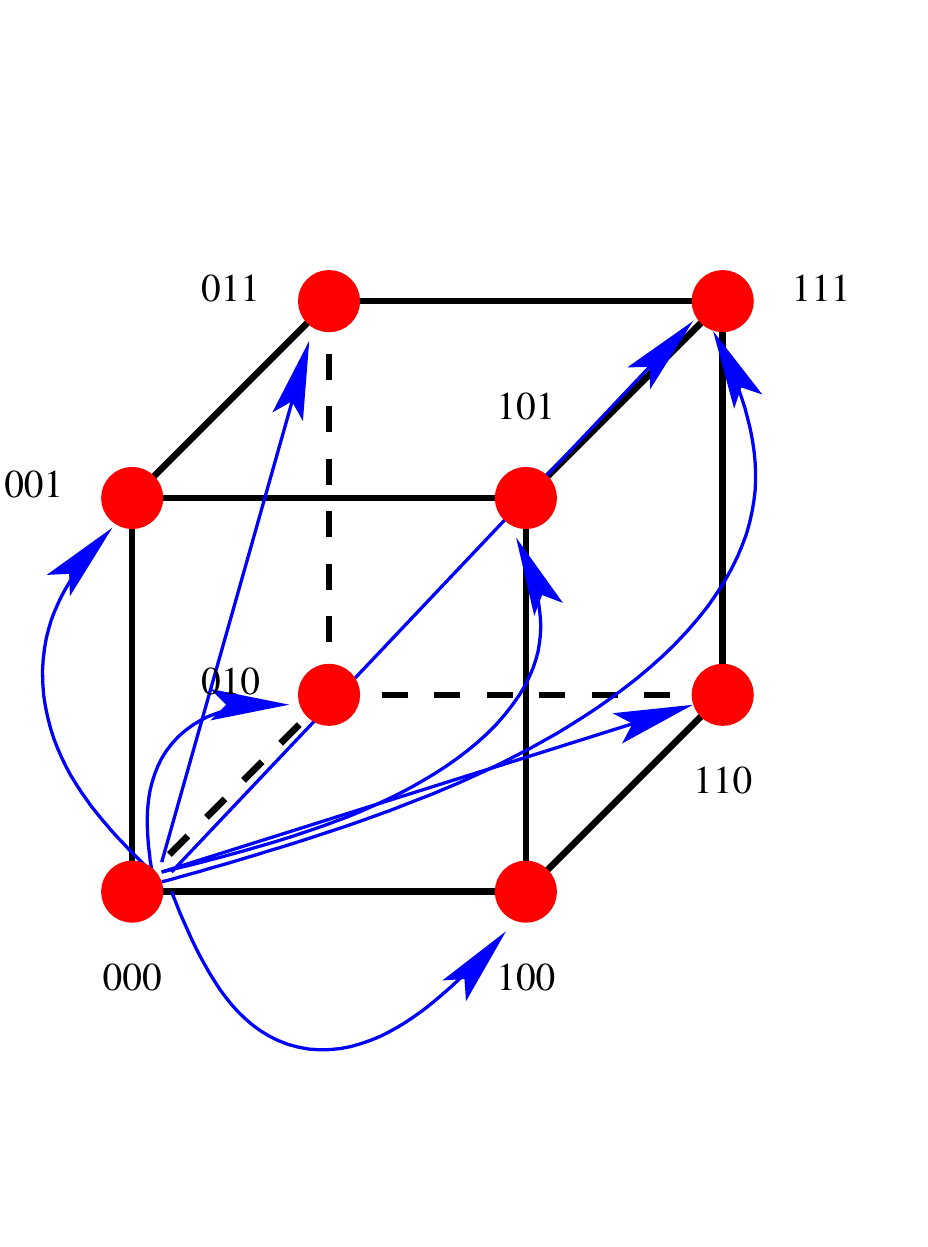}}
\caption{The 3-dimensional  vector space over \( \mathbb F_2 \) with the $7$ lines through the origin illustrated}
\end{figure}

The ambient space,  the analog of $\mathbb R^3$, becomes  a set of $8$ points, which may be viewed as the vertices of a cube, or the points of the three dimensional vector space over the two element field, \(\mathbb F_2^3 \). A \emph{line} is an abstract segment that connects two of these points. Thus there are \( \binom{8}{2}=28 \) lines. A \emph{line complex} is a collection of $8$ lines. Hence, a  line complex has as many  lines as there are points in $\mathbb F_2^3$ and  there are \( \binom{28}{8}= 3,108,105 \) line complexes.   For an unknown function $f(x)$ from the points of $\mathbb F_2^3$ to  $\mathbb R$  (or to $\mathbb C$) , ``integrals" along a lines simply become sums of function values on the points comprising  these lines. Thus, we  aim to recover an unknown function $f(x)$ from its integrals along lines. A  celebrated principle of Ethan Bolker \cite{Bol}, \emph{The Bolker Condition}, implies that a  function is determined by its  integral along all lines. Interpreting this in the context  of linear algebra in a finite dimensional vector space we see that not all $28$ lines are required, and that  some collections of $8$ lines, i.e., line complexes, suffice; these are called \emph{admissible} line complexes, after  I.M. Gel'fand \cite{Gel1,Gel2, Kir}. In the  finite category one cannot recover data  which is omitted from measurements. Thus, in order  to be admissible, a line complex must not ``omit" points, i.e. the union of the lines must include all points in the ambient space. More generally, it is easy to see \cite{Gr2} that an admissible complex $\cal C$ must have the three properties below, and it is not  difficult to see that these properties  suffice for admissibility:

\begin{itemize}
\item $\cal C$ omits no point, i.e., each of the $8$ points is included in some line in $\cal C$
\item $\cal C$ has no isolated subtrees
\item $\cal C$ contains no even cycles
\end{itemize}

The problem of admissibility is to describe (and count) the line  complexes satisfying these conditions. As  the Radon transform on $\mathbb F_2^3$ can be modeled by a $28 \times 8$ incidence matrix $M$ (rows representing lines, columns representing points), this problem is equivalent to describing  which $8 \times  8$ submatrices of $M$ are of maximal rank.  A short, brute  force program written in the computer   \emph{GNU Octave} is presented in \cite{Gr2}. Curiously, running  this program twice yielded two different results: counts of $937,438$ and  $937,440$ admissible  complexes. While we normally expect a program to produce the same answer in repeated runs, there are contexts where different results are expected, e.g., in compiling a {\LaTeX}  file with references. Thus, we were curious to  see if we could perform the count ``by hand" and, along the way,  produce an  illustrated scrapbook of examples.

\par
In making the count, we  chose to explore the complementary, and larger, class of \emph{inadmissible} complexes and, at the same  time, to give rather concrete  descriptions of the configurations that occur.  The method of proof  does not meet the criteria outlined by Leslie Lamport in  \emph{How to Write a $21^{st}$ Century Proof} \cite{La}, but is ripe  for conversion into one that does. Moreover, it is clear  that a computer can be taught to make the counts and generate examples and illustrations systematically and in larger contexts \cite{NS}. In such contexts the   scrapbook  below may be a  useful debugging tool. 

\par
For a larger finite  (projective) model of tomography (billions rather than millions of line  complexes), see \cite{FG}. While computation is a useful discovery tool in that context, it became dispensable after discovery, as the projective geometry makes analysis by hand quite straightforward.

\bigskip\noindent
Below we will perform the count  and illustration of inadmissible complexes by first considering complexes that omit points, then introduce isolated trees, and finally address even cycles.  Inclusion-exclusion will be the order of the day.

\section{Complexes that omit one or more points}

\subsection{Complexes that omit one point}

\noindent
First we enumerate complexes that are ``missing points'', that is, complexes $\cal C$ such that there exist points $p \in \mathbb F_2^3$ so that no line $\ell \in \cal C$ passes through $p$.
Many of these exist. 

There are seven lines through a point $p$, so the complexes that miss $p$ have $8$ lines chosen from the $28-7=21$. Now $\binom{21}{8}=203,440$. Multiplying this by the number of points, $8$, and accounting for  double counting (because there are complexes that omit more than one point) we obtain:

\begin{lemma} There are $\binom{21}{8} \cdot 8 = 1,627,920$ complexes that omit points. Here, each complex is counted with multiplicity equal to the number of points in $\mathbb F_2^3$ which it misses. \end{lemma}

\subsection{Complexes that omit two or more points}

How many complexes miss two points? There are $7+7-1=13$ lines through one or the other or both points. So a complex that misses both points has $8$ lines chosen from among $28-13=15$ lines. There are $28$ pairs of points, so we have double counted $28 \cdot \binom{15}{8}=28 \cdot 6,435
= 180,180$ complexes. (Note that we have double counted the double counting, because there are complexes that miss three points.)

\begin{lemma} The number of complexes that omit a pair of points is  

$$
28 \cdot \binom{15}{8}=28 \cdot 6,435 = 180,180.
$$ 

Here each complex is counted with multiplicity equal to the number of pairs of points that it misses.
\end{lemma}

\subsection{Complexes that omit three or more points}

How many lines pass through one or more of three given 
points? All but the $10$ that form the complete graph on the remaining $5$ points. Thus, to exhibit all complexes omitting three or more points, choose three points from $8$ and then choose $8$ lines from among the $10$ lines avoiding these points.  Thus we have: 

\begin{lemma} The number of complexes that omit precisely three points is
$\binom{10}{8} \cdot \binom{8}{3} = 2,520.$ There are no line complexes that miss four or more points.
\end{lemma}

Putting the above lemmas together we have

\begin{lemma}
The number of complexes that avoid (omit) one or more points is:
$$
1,627,920 - 180,180 + 2,520 = 1,450,260.
$$ 
This count is without multiplicity.
\end{lemma}

\section{Complexes with isolated lines}

\subsection{Complexes with one or more isolated lines}

Another type of non-admissible complex is one where a single line $\ell$ is `isolated', i.e., meets no other line in the complex. (This is the simplest case of an isolated tree.)  

\centerline{
\begin{tikzpicture}
 \Vertex[x=0,style={shading=ball}, NoLabel]{A} 
 \Vertex[x=2,style={shading=ball}, NoLabel]{B}
 \Vertex[x=0.3, y=2,style={shading=ball}, NoLabel]{C}
  \Vertex[x=2.3, y=2,style={shading=ball}, NoLabel]{D}
 \Vertex[x=3.2, y=3.0,style={shading=ball}, NoLabel]{E}
 \Vertex[x=4,style={shading=ball}, NoLabel]{F}
 \Vertex[x=4.3, y=2,style={shading=ball}, NoLabel]{G}
  \Vertex[x=5.8, y=2,style={shading=ball}, NoLabel]{H}
  \Edge(A)(B)   \Edge(A)(C) \Edge(C)(D)   \Edge(B)(D)
  \Edge(D)(E) 
  \Edge(E)(G) \Edge(D)(G)
  \Edge (F)(H)
  \end{tikzpicture}  
  }

How many of these are there? Well, how many lines meet $\ell$? Precisely $7+7-1=13=28-15$ lines meet $\ell$. So the number of complexes having $\ell$ as an isolated line is $\binom{15}{7}= 6,435.
$ Accounting for each of the $28$ lines, with the usual double counting reminder, we have 

\begin{lemma} There are
$6,435 \cdot  28= 180,180 $ complexes with one or more isolated lines. 
Each complex is counted with multiplicity equal to the number of isolated lines it has.
\end{lemma}

\subsection{Complexes with two or more disjoint isolated lines} {\quad}

\medskip

\begin{figure}[H]
\begin{tikzpicture}
 \Vertex[x=0,style={shading=ball}, NoLabel]{A} 
 \Vertex[x=2,style={shading=ball}, NoLabel]{B}
 \Vertex[x=0.3, y=2,style={shading=ball}, NoLabel]{C}
  \Vertex[x=2.3, y=2,style={shading=ball}, NoLabel]{D}
 \Vertex[x=3.2, y=3.0,style={shading=ball}, NoLabel]{E}
 \Vertex[x=4,style={shading=ball}, NoLabel]{F}
 \Vertex[x=4.3, y=2,style={shading=ball}, NoLabel]{G}
  \Vertex[x=5.8, y=2,style={shading=ball}, NoLabel]{H}
  \Edge(A)(B)   \Edge(A)(C) \Edge(C)(D)   \Edge(B)(D)
  \Edge(E)(G)   \Edge (F)(H) \Edge(A)(D) \Edge(B)(C)
  \end{tikzpicture}  
\end{figure}

\smallskip

\medskip\noindent
If $\ell$ is a line, there are $13$ lines meeting $\ell$ and $15$ lines disjoint from $\ell$. Thus there are $(28)(15)/2=210$ pairs of disjoint lines. Given a complex with a pair of disjoint lines, the other $6$ lines of the complex must form the complete graph on the remaining four points. Thus there are $210$ complexes with precisely two disjoint isolated lines. Clearly a complex cannot have three disjoint isolated lines.

\begin{lemma} 
There are $(28)(15)/2=210$ complexes with precisely two isolated lines, and there are no complexes with three or more isolated lines.
\end{lemma}

\begin{lemma}
There are $180,180-210=179,970$ complexes with one or more isolated lines. These complexes are counted without multiplicity.
\end{lemma}

\pagebreak

\section{Complexes with both omitted points and isolated lines}
\subsection{Complexes with one or more isolated lines and one or more omitted points} {\quad}

\begin{figure}[H] \begin{tikzpicture}
 \Vertex[x=0,style={shading=ball}, NoLabel]{A} 
 \Vertex[x=2,style={shading=ball}, NoLabel]{B}
 \Vertex[x=0.3, y=2,style={shading=ball}, NoLabel]{C}
  \Vertex[x=2.3, y=2,style={shading=ball}, NoLabel]{D}
 \Vertex[x=3.2, y=3.0,style={shading=ball}, NoLabel]{E}
 \Vertex[x=4,style={shading=ball}, NoLabel]{F}
 \Vertex[x=4.3, y=2,style={shading=ball}, NoLabel]{G}
  \Vertex[x=5.8, y=2,style={shading=ball}, NoLabel]{H}
  \Edge(A)(B)   \Edge(A)(C) \Edge(C)(D)   \Edge(B)(D)
  \Edge(D)(E) \Edge(B)(C)
  \Edge (F)(H) 
  \Edge(C)(E)
  \end{tikzpicture}  \end{figure}

There are five points that meet neither the designated omitted point nor the isolated line, hence there are $\binom{5}{2}=10$  ``permissibile" lines. We must choose $7$ lines among these to form a complex, and there are $8 \cdot 28$ point-line pairs.

\begin{lemma} There are no  complexes with one isolated line and two omitted points. 
\end{lemma}

\begin{proof}
The complement of the union of the omitted points and the isolated line has $4$ points, and these form $6$ lines, not enough to form a line complex. 
\end{proof}

\begin{lemma}
There are no complexes with two disjoint isolated lines and an omitted point
\end{lemma}

\begin{proof}
There are five points in the union of the two lines and point, hence three points left, not enough to span a line complex.
\end{proof}

\begin{lemma}
The number of complexes with one isolated line and one omitted point is 
$(8 \cdot 21) \binom{10}{7}=20,160$.  This count is multiplicity free.
\end{lemma}

\begin{proof}
There are $8 \cdot 21=168$ disjoint point-line pairs (or $28 \cdot 6=168$ disjoint line-point pairs). Given a disjoint point-line pair there are $5$ remaining points and $\binom{5}{2}=10$ lines in their complete graph. Of these we must choose $7$ to obtain  a line complex. Because of the preceeding lemmas, there can be no additional isolated points, nor additional isolated lines, hence there are no multiplicities here. 
\end{proof}

\medskip


\section{Complexes with isolated trees and omitting no points} {\quad}

\begin{figure}[H] \begin{tikzpicture}
 \Vertex[x=0,style={shading=ball}, NoLabel]{A} 
 \Vertex[x=2,style={shading=ball}, NoLabel]{B}
 \Vertex[x=0.3, y=2,style={shading=ball}, NoLabel]{C}
  \Vertex[x=2.3, y=2,style={shading=ball}, NoLabel]{D}
 \Vertex[x=3.2, y=3.0,style={shading=ball}, NoLabel]{E}
 \Vertex[x=4,style={shading=ball}, NoLabel]{F}
 \Vertex[x=4.3, y=2,style={shading=ball}, NoLabel]{G}
  \Vertex[x=5.8, y=2,style={shading=ball}, NoLabel]{H}
  \Edge(A)(B)   \Edge(A)(C) 
  \Edge(C)(D)   
  \Edge(D)(E) \Edge(B)(C)
  \Edge (F)(H) 
  \Edge(C)(E) \Edge(G)(H)
  \end{tikzpicture}  \end{figure}

\begin{lemma} The number of complexes that contain isolated trees and omit no points is $200,970$. These complexes are counted  without multiplicity.
\end{lemma}

\begin{sublemma}
The number of complexes that omit no point and contain at least one isolated line is 
 \begin{equation}\label{159810}
\binom{8}{2} \left[  \binom{15}{7} - \binom {6}{1} \binom{10}{7} - \frac{1}{2} \binom{6}{2} \binom {6}{6} \right] = 159,810.
\end{equation}
These complexes are counted without multiplicity.
\end{sublemma}

\begin{proof}

First choose an isolated line $\ell$ by choosing $2$ points from $8$. 

This explains the lefmost factor in the above count \eqref{159810}.

There are $6$ remaining points. 

There are $\binom{6}{2}=15$ lines among these $6$ points, and we need to choose $7$ lines from these to complement the isolated line and form a complex. 

There are $\binom{15}{7}$ ways to do this, but some of these ways may omit a point.

(They cannot omit more than one point, for if they were to do so, at most $4$ points would remain and amongst them at most $6$ lines, which are too few.) 

In how many ways can they omit a point $p$? 

Choose a point $p$ disjoint from $\ell$. (There are $6$ such points). 

Then choose $7$ of the $\binom{5}{2}=10$ lines connecting the $5$ remaining points. The $7$ lines cannot omit  any of the $5$ points, since $4$ points have only  $6$ lines among them.

Hence we have enumerated $6 \cdot \binom{10}{7}$ complexes with $\ell$ as an isolated line and with an omitted point. 

Thus we have \( \binom{15}{7} - \binom {6}{1} \binom{10}{7} \) complexes with $\ell$ as isolated line with no omitted point.
But some of these complexes may have an additional disjoint line, and these are counted twice. (Such complexes cannot have isolated points, lest there not be enough additional lines to form a complex.)

In how many ways can we produce a complex with two disjoint lines here? 

Choose $2$ of the $6$ points in the complement of $\ell$, select the line through them, and then select $6$ of the $6$ lines through the remaining $4$ points.  This accounts for the $\frac{1}{2} \binom{6}{2} \binom {6}{6}$ term.

These complexes are counted without multiplicity, as promised.
\end{proof}

\begin{sublemma}
The number of complexes with an isolated $3$-point tree and no omitted points is: 
$$
\binom{8}{3} \binom{3}{1} \left[ \binom{10}{6} - \binom{5}{1} \binom{6}{6}
\right] = 34,440.
$$
\end{sublemma}

\begin{proof}
Choose $3$ from $8$ points to form the vertices of the isolated tree. Then choose from these the one that has valence $2$. (The others will have valence $1$.) We must now choose $6$ lines amonst the remaining $5$ points. Of the $\binom{10}{6}$ ways to do this, some omit a point. There are $\binom{5}{1}$ possible omitted points and all $\binom{6}{6}$ lines amongst the remaining points must be added to form a complex.
\end{proof}

\begin{sublemma}
There number of complexes with an isolated $4$-point tree is
$$
\binom{8}{4} \left[ \frac{4!}{2} + \binom{4}{1} \right] \binom{6}{5}= 6,720,
$$
counted without multiplicity.
\end{sublemma}

\begin{figure}[H] \begin{tikzpicture}
 \Vertex[x=0,style={shading=ball}, NoLabel]{A} 
 \Vertex[x=2,style={shading=ball}, NoLabel]{B}
 \Vertex[x=0.3, y=2,style={shading=ball}, NoLabel]{C}
  \Vertex[x=2.3, y=2,style={shading=ball}, NoLabel]{D}
 \Vertex[x=3.2, y=3.0,style={shading=ball}, NoLabel]{E}
 \Vertex[x=4,style={shading=ball}, NoLabel]{F}
 \Vertex[x=4.3, y=2,style={shading=ball}, NoLabel]{G}
  \Vertex[x=5.8, y=2,style={shading=ball}, NoLabel]{H}
  \Edge(A)(B)   \Edge(A)(C) 
  \Edge(C)(D)  
  \Edge(B)(D)
   \Edge(E)(G)
\Edge(B)(C)
  \Edge (F)(H) 
 \Edge(G)(H)
  \end{tikzpicture}  \end{figure}

It will be noted that automatically such a complex omits no points.

\begin{proof}
Choose $4$ points from $8$ to form the vertices of an isolated tree. There are two possible topologies for this tree, one linear and the other not. In the linear case we choose one of $4!$ possible orderings of the vertices of the tree, and divide by two to account for the possible orientations. In the nonlinear case, one of the $4$ points must have valence $3$, and the others have valence $1$, respectively. Hence there are $\binom{4}{1}$ nonlinear isolated trees. 

Once the isolated tree is fixed, we must choose $5$ from amongst the $6$ possible lines connecting the remaining $4$ points, with a count of $\binom{6}{5}$, with no omitted points possible.  As we can  have at most one isolated $4$-point tree in a complex (else there would be fewer than $8$ lines),  the count is  without multiplicity. 
\end{proof}

We add the results of the sublemmas to prove the Lemma.

\section{\emph{Proper Complexes}:  { Complexes omitting no point, with no isolated trees.}}

\begin{lemma}
The number of complexes that omit no point and contain no isolated trees is
$$
1,657,845-200,970=1,456,875.
$$
These complexes are counted without multiplicity.
\end{lemma}
Some of these complexes contain even cycles, and we must identify and enumerate them. We will propose lemmas for the possible cardinalities of the cycles: $8,6,$ and $4$. 

\begin{figure}[H] \begin{tikzpicture}
 \Vertex[x=0,style={shading=ball}, NoLabel]{A} 
 \Vertex[x=2,style={shading=ball}, NoLabel]{B}
 \Vertex[x=0.3, y=2,style={shading=ball}, NoLabel]{C}
  \Vertex[x=2.3, y=2,style={shading=ball}, NoLabel]{D}
 \Vertex[x=3.2, y=3.0,style={shading=ball}, NoLabel]{E}
 \Vertex[x=4,style={shading=ball}, NoLabel]{F}
 \Vertex[x=4.3, y=2,style={shading=ball}, NoLabel]{G}
  \Vertex[x=5.8, y=2,style={shading=ball}, NoLabel]{H}
  \Edge(A)(C)   
  \Edge(C)(D)  \Edge(D)(E)  \Edge(E)(G)      \Edge(G)(H) \Edge(H)(F)
  \Edge(F)(B) \Edge(B)(A) 
  \end{tikzpicture}  \end{figure}

\begin{definition} Henceforth we will call a line complex {\bf proper} if it omits no point and has no isolated trees. 
\end{definition}

\noindent
All counts given henceforth are without multiplicity,  tacitly.

\subsection{Proper Complexes containing a $6$-cycle or an $8$-cycle}

\begin{lemma}
There are \( \frac{28 \cdot 6!}{8}= 2,520 \) line complexes containing an $8$-cycle.
\end{lemma}
\begin{proof}
Choose a ``base" line $\ell = \overline{p_1p_2}$ for the cycle. Choose one of the remaining points $p_3$ to form two lines: $\overline{p_1p_2}$ and $\overline{p_2p_3}$. Continue to choose points $p_4, \ldots p_8$ to form the $8$-cycle 
\( \overline{p_1p_2}, \overline{p_2p_3}, \ldots , \overline{p_7p_8} \). 
We have $28$ choices for the first line $\ell$ and \( 6 \times 5 \times \ldots \times 2 \times 1 =6! \) for the collection of points $p_3, \ldots , p_8$. But our starting line could have been any of the $8$ lines of the $8$-cycle, so the total number of $8$-cycles is \( \frac{28 \cdot 6!}{8}= 2,520 \) .

Here is an alternative way to count There are $8!$ ways to order linearly the $8$ points in our space. Order the $8$ points, start with the first point $P$ and march along, forming lines, then glue the last point to $P$, the ``gluing point", forming a cycle. Allowing for the $8$ possible gluing  points on our cycle and two possible orientations we obtain \( \frac{ 8!} {2 \cdot 8}=2520 \),  the result declared. 
\end{proof}

\begin{lemma} The number of proper complexes containing a $6$-cycle is
\[
\binom{8}{6} \frac{ 6! } {2 \cdot 6} 
\left[ (1+2) \binom{6}{1}+2 \binom{6}{2} \right]
= 80,640
.
\]
These complexes are counted without multiplicity.
\end{lemma}

\begin{tabular}{|c|c|}
\hline &  \\ \qquad
\begin{tikzpicture}
 \Vertex[x=0,style={shading=ball}, NoLabel]{A} 
 \Vertex[x=2,style={shading=ball}, NoLabel]{B}
 \Vertex[x=0.3, y=2,style={shading=ball}, NoLabel]{C}
  \Vertex[x=2.3, y=2,style={shading=ball}, NoLabel]{D}
 \Vertex[x=3.2, y=3.0,style={shading=ball}, NoLabel]{E}
 \Vertex[x=4,style={shading=ball}, NoLabel]{F}
 \Vertex[x=4.3, y=2,style={shading=ball}, NoLabel]{G}
  \Vertex[x=5.8, y=2,style={shading=ball}, NoLabel]{H}
  \Edge(B)(D)    
  \Edge(C)(D)  \Edge(D)(E)  \Edge(E)(G)      \Edge(G)(H) \Edge(H)(F)
  \Edge(F)(B) \Edge(B)(A) 
  \end{tikzpicture} \qquad
 & \qquad
\begin{tikzpicture}
 \Vertex[x=0,style={shading=ball}, NoLabel]{A} 
 \Vertex[x=2,style={shading=ball}, NoLabel]{B}
 \Vertex[x=0.3, y=2,style={shading=ball}, NoLabel]{C}
  \Vertex[x=2.3, y=2,style={shading=ball}, NoLabel]{D}
 \Vertex[x=3.2, y=3.0,style={shading=ball}, NoLabel]{E}
 \Vertex[x=4,style={shading=ball}, NoLabel]{F}
 \Vertex[x=4.3, y=2,style={shading=ball}, NoLabel]{G}
  \Vertex[x=5.8, y=2,style={shading=ball}, NoLabel]{H}
  \Edge(B)(D)    
  \Edge(C)(D)  \Edge(D)(E)  \Edge(E)(G)      \Edge(G)(H) \Edge(H)(F)
  \Edge(F)(B)  \Edge(D)(A)
  \end{tikzpicture}  \qquad
  \\ \hline 
  
 & 
 \begin{tikzpicture}
 \Vertex[x=0,style={shading=ball}, NoLabel]{A} 
 \Vertex[x=2,style={shading=ball}, NoLabel]{B}
 \Vertex[x=0.3, y=2,style={shading=ball}, NoLabel]{C}
  \Vertex[x=2.3, y=2,style={shading=ball}, NoLabel]{D}
 \Vertex[x=3.2, y=3.0,style={shading=ball}, NoLabel]{E}
 \Vertex[x=4,style={shading=ball}, NoLabel]{F}
 \Vertex[x=4.3, y=2,style={shading=ball}, NoLabel]{G}
  \Vertex[x=5.8, y=2,style={shading=ball}, NoLabel]{H}
  \Edge(B)(D)    
  \Edge(C)(D)  \Edge(D)(E)  \Edge(E)(G)      \Edge(G)(H) \Edge(H)(F)
  \Edge(F)(B)  \Edge(C)(A)
  \end{tikzpicture}  \qquad
 \\ \hline
\end{tabular}

\begin{proof}
We choose $6$ of $8$ points to form the $6$-cycle. (It is easy to see that an admissible complex cannot have more than one $6$-cycle whilst omitting no point, for, starting with a given $6$-cycle of lines, the remaining two lines of the complex must ``reach out" to the remaining two points, and therefore cannot closeup and cannot form additional cycles.) As in the count of $8$-cycles above, there are $\frac{6!}{6 \cdot 2}$ ways to form a $6$-cycle given its $6$ vertices.

The remaining two points must be attached to the $6$-cycle. Since only $2$ lines remain to be chosen, each of the attachments of the $2$ aforementioned points to the $6$-cycle must be via a unique line. The attachments can occur through a single point on the $6$-cycle, or through two different points. 

In the single attachment point case, we choose the attachment point, $A$. Either both remaining points are attached to $A$, or else one provides a ``life-line" for the other. There are are $1+2=3$ ways for the two remaining points to be attached to $A$. 

In the $2$-point attachment case, we choose $2$ of the $6$ cycle points for attachment. The two remaining points may be attached to the attachment points in $2$ ways.
\end{proof}

\subsection{Proper Complexes containing a $4$-cycle} {\quad}

\smallskip \,

We continue now with the case of a $4$-cycle. This is the most involved and, some would say, the most perplexing and frustrating part of the analysis. By analogy with Grassmann manifolds, this situation is like the ``middle Grassmannian" $Gr(k,2k)$, which enjoys the  richest structure in its family.

\pagebreak

\begin{lemma}
If a proper line complex has more than one $4$-cycle, then the cycles must be disjoint.
\end{lemma}

\begin{figure}[H] \begin{tikzpicture}
 \Vertex[x=0,style={shading=ball}, NoLabel]{A} 
 \Vertex[x=2,style={shading=ball}, NoLabel]{B}
 \Vertex[x=0.3, y=2,style={shading=ball}, NoLabel]{C}
  \Vertex[x=2.3, y=2,style={shading=ball}, NoLabel]{D}
 \Vertex[x=3.2, y=3.0,style={shading=ball}, NoLabel]{E}
 \Vertex[x=4,style={shading=ball}, NoLabel]{F}
 \Vertex[x=4.3, y=2,style={shading=ball}, NoLabel]{G}
  \Vertex[x=5.8, y=2,style={shading=ball}, NoLabel]{H}
  \Edge(A)(B)   \Edge(A)(C) 
  \Edge(C)(D)  
  \Edge(B)(D)
   \Edge(E)(G) \Edge(E)(F)
  \Edge (F)(H) 
 \Edge(G)(H)
  \end{tikzpicture}  \end{figure}

\begin{proof} Let $F$ be a $4$-cycle in the complex.
Can there be a $4$-cycle $G \ne F$ in the complex which intersects $F$? Suppose so. Then $F\cup G$ is not everything and $F,G$ must have a line in common, else the complex has no additional lines and has points not in $F \cup G$, hence omitted points. Thus $F,G$ must have at least one line in common.

Suppose $F,G$ have exactly one line in common and no other points in common. Then $F \cup G$ contains $7$ lines but involves only $6$ points, so the remaining line cannot include both remaining points and, at the same time, be connected to the rest of the complex. So $F$ and $G$ must have  $3$ or $4$ points in common. 

If $F$ and $G$ have $3$ points in common, they can have at most $2$ lines in common. (If they had $3$ lines in common and only $3$ points in common, they would have to have a triangle in common, and hence could not be $4$-cycles.) Then $F$ and $G$ involve $5$ points and have at least $6$ lines amongst them. Only two lines remain to be chosen to cover the remaining $3$ points at large, and this can only be done if these two lines form a figure $L$, or if one line ``hangs" off one or more of the $4$-cycles, and the other line is isolated; in either case we have an isolated tree, hence forbidden.

If $F$ and $G$ have $4$ points in common, both cycles are subgraphs of the complete graph on these four points. As $F,G$ are distinct, they must have two lines in common: if they had three lines  in common they'd  have to share a fourth line  also, to complete the $4$-cycle. Thus $F,G$ have $6$ lines together. The remaining two lines of the complex must go through  the four points of the complement of $F \cup G$, and they can only do so  by being isolated lines.
\end{proof}

\begin{figure}[H] \begin{tikzpicture}
 \Vertex[x=0,style={shading=ball}, NoLabel]{A} 
 \Vertex[x=2,style={shading=ball}, NoLabel]{B}
 \Vertex[x=0.3, y=2,style={shading=ball}, NoLabel]{C}
  \Vertex[x=2.3, y=2,style={shading=ball}, NoLabel]{D}
 \Vertex[x=3.2, y=3.0,style={shading=ball}, NoLabel]{E}
 \Vertex[x=4,style={shading=ball}, NoLabel]{F}
 \Vertex[x=4.3, y=2,style={shading=ball}, NoLabel]{G}
  \Vertex[x=5.8, y=2,style={shading=ball}, NoLabel]{H}
  \Edge(A)(B)   \Edge(A)(C)  \Edge(A)(D)
  \Edge(C)(D)  
  \Edge(B)(D)
   \Edge(E)(G)
\Edge(B)(C)
  \Edge (F)(H) 
  \end{tikzpicture}  \end{figure}

\begin{lemma}
The number of proper line complexes containing more than one $4$-cycle is
$$
\binom{8}{4} \left(\frac{4!}{4 \cdot 2} \right) \left(\frac{4!}{4 \cdot 2} \right)
\left( \frac{1}{2} \right) = 315.
$$
\end{lemma}
\begin{proof}
Choose $4$ points. Arguing as in the case of an $8$-cycle, there are 
$\frac{4!}{4 \cdot 2}$ ways to form a cycle with these points. Then we form a $4$-cycle with the remaining $4$ points, and then account for the transposition of the two $4$-cycles.
\end{proof}

\noindent
\begin{remark}
In the lemmas below we'll fix a $4$-cycle, say $T$ and speak of a ``unique $4$-cycle". In making the final count we will need to multiply the lemma counts by the number of possible $4$-cycles in the complex.  The number of ways to choose $4$ points among $8$ is \( \binom {8}{4}=70 \). Given $4$ points, there are $3$ ways to make an unoriented $4$-cycle from them. Thus the counts in the lemmas below will be multiplied by \( 70 \cdot 3=210 \) in the final count.
\end{remark}

\begin{lemma} The number of topologically disconnected proper line complexes containing a unique $4$-cycle is:
\begin{equation}\label{5880}
\binom{8}{4} \left(\frac{4!}{4 \cdot 2} \right) 
\binom{4}{1} \binom{7}{1}
=5,880.
\end{equation}
\end{lemma}

\begin{figure}[H]
\begin{tikzpicture}
 \Vertex[x=0,style={shading=ball}, NoLabel]{A} 
 \Vertex[x=2,style={shading=ball}, NoLabel]{B}
 \Vertex[x=0.3, y=2,style={shading=ball}, NoLabel]{C}
  \Vertex[x=2.3, y=2,style={shading=ball}, NoLabel]{D}
 \Vertex[x=3.2, y=3.0,style={shading=ball}, NoLabel]{E}
 \Vertex[x=4,style={shading=ball}, NoLabel]{F}
 \Vertex[x=4.3, y=2,style={shading=ball}, NoLabel]{G}
  \Vertex[x=5.8, y=2,style={shading=ball}, NoLabel]{H}
  \Edge(A)(B)   \Edge(A)(C) \Edge(C)(D)   \Edge(B)(D) 
  \Edge(E)(G)
  \Edge (F)(H) \Edge(F)(G) \Edge(G)(H)
  \end{tikzpicture}  
\end{figure}

\begin{proof}
The $4$ points not included in the $4$-cycle cannot form a tree, nor a $4$-cycle. Therefore they must contain a $3$-cycle. We choose $1$ point from $4$ as the point not in the $3$-cycle and connect this point to one of the other seven points by a line. This determines the line complex entirely, since the other $3$ points must be connected by their complete graph. 

Reading  \eqref{5880} from left to right, the first two factors count the number of way to select  a $4$-cycle, the next factor counts selection of a vertex not in a $3$-cycle and the last factor selects a line to connect this vertex to the rest of the complex.
\end{proof}

We now examine topologically connected complexes containing a unique $4$-cycle. The $4$-cycle must have at least one vertex of valence greater than $2$, or else the complex would not be connected.

\begin{lemma}
The number of connected proper line complexes with a unique $4$-cycle with precisely one of its vertices of valence greater than $2$ is

\begin{equation}\label{500}
\binom{4}{1} 
\left[
\binom{4}{4}+ \binom{4}{3} \cdot 3+\binom{4}{2} \left( \binom{5}{2} -2 \right) +\binom{4}{1} \left( \binom{6}{3} -4 \right)
\right]
= 500.
\end{equation}

\end{lemma} 
\begin{proof}
We parse the inner summands of \eqref{500} above from left to right.
Choose a $4$-cycle as before and mark one vertex $D$ to have valence greater than $2$. 

\begin{figure}[H]
\begin{tikzpicture}
 \Vertex[x=0,style={shading=ball}, NoLabel]{A} 
 \Vertex[x=2,style={shading=ball}, NoLabel]{B}
 \Vertex[x=0.3, y=2,style={shading=ball}, NoLabel]{C}
  \Vertex[x=2.3, y=2,style={shading=ball}, 
  ]{D}
 \Vertex[x=3.2, y=3.0,style={shading=ball}, NoLabel]{E}
 \Vertex[x=4,style={shading=ball}, NoLabel]{F}
 \Vertex[x=4.3, y=2,style={shading=ball}, NoLabel]{G}
  \Vertex[x=2.8, y=0.4,style={shading=ball}, NoLabel]{H}
  \Edge(A)(B)   \Edge(A)(C) \Edge(C)(D)   \Edge(B)(D)
  \Edge(D)(E)   \Edge(D)(F)
  \Edge(D)(G)
 \Edge (D)(H)
  \end{tikzpicture}  
\end{figure}

If vertex $D$ has valence $6$, there is only $1=\binom{4}{4}$ way for all $4$ of the remaining points to be connected to $D$. 

\begin{figure}[H]
\begin{tikzpicture}
 \Vertex[x=0,style={shading=ball}, NoLabel]{A} 
 \Vertex[x=2,style={shading=ball}, NoLabel]{B}
 \Vertex[x=0.3, y=2,style={shading=ball}, NoLabel]{C}
  \Vertex[x=2.3, y=2,style={shading=ball}, 
  ]{D}
 \Vertex[x=3.2, y=3.0,style={shading=ball}, NoLabel]{E}
 \Vertex[x=4,style={shading=ball}, NoLabel]{F}
 \Vertex[x=4.3, y=2,style={shading=ball}, NoLabel]{G}
  \Vertex[x=2.8, y=0.4,style={shading=ball}, NoLabel]{H}
  \Edge(A)(B)   \Edge(A)(C) \Edge(C)(D)   \Edge(B)(D)
  \Edge(D)(E)   \Edge(D)(F)
  \Edge(D)(G)
 \Edge (F)(H)
  \end{tikzpicture}  
\end{figure}

If $D$ has valence $5$, choose $3$ of the $4$ remaining points to be connected to $D$, and then choose one of these $3$ to connect to the last remaining point.

\begin{figure}[H]
\begin{tikzpicture}
 \Vertex[x=0,style={shading=ball}, NoLabel]{A} 
 \Vertex[x=2,style={shading=ball}, NoLabel]{B}
 \Vertex[x=0.3, y=2,style={shading=ball}, NoLabel]{C}
  \Vertex[x=2.3, y=2,style={shading=ball}, 
  ]{D}
 \Vertex[x=3.2, y=3.0,style={shading=ball}, 
 ]{E}
 \Vertex[x=4,style={shading=ball}, 
 ]{F}
 \Vertex[x=4.3, y=2,style={shading=ball}, 
 ]{G}
  \Vertex[x=2.8, y=0.4,style={shading=ball}, 
  ]{H}
  \Edge(A)(B)   \Edge(A)(C) \Edge(C)(D)   \Edge(B)(D)
  \Edge(D)(E)   
  \Edge(D)(G)
\Edge (F)(H) \Edge (G)(F)
  \end{tikzpicture}  
\end{figure}

If $D$ has valence $4$, choose $2$ of the remaining $4$ points, $E,G,$ to connect to $D$. Call the remaining two points $F,H$. We must choose $2$ lines from the $6$ in the complete graph on $EFGH$. But $EG$ is forbidden (else an isolated tree or an omitted point results), so only $2$ of $5$ lines are available, and we cannot choose both to go through $F$ and omit $H$ nor both to go through $H$ and omit $F$, so we have $\binom{5}{2}-2=8$ choices. Hence there are 
$\binom{4}{2} \cdot \left( \binom{5}{2}-2\right) =48$ complexes with $D$ of valence $4$.

\pagebreak

\begin{figure}[H]
\begin{tabular}{|l|r|}
\hline &  \\
\begin{tikzpicture}
 \Vertex[x=0,style={shading=ball}, NoLabel]{A} 
 \Vertex[x=2,style={shading=ball}, NoLabel]{B}
 \Vertex[x=0.3, y=2,style={shading=ball}, NoLabel]{C}
  \Vertex[x=2.3, y=2,style={shading=ball}, NoLabel]{D}
 \Vertex[x=3.2, y=3.0,style={shading=ball}, NoLabel]{E}
 \Vertex[x=4,style={shading=ball}, NoLabel]{F}
 \Vertex[x=4.3, y=2,style={shading=ball}, NoLabel]{G}
  \Vertex[x=5.8, y=2,style={shading=ball}, NoLabel]{H}
  \Edge(A)(B)   \Edge(A)(C) \Edge(C)(D)   \Edge(B)(D)
  \Edge(D)(E) \Edge(F)(G)   \Edge(G)(H) \Edge (F)(H)
  \end{tikzpicture} 
&
  \begin{tikzpicture}
 \Vertex[x=0,style={shading=ball}, NoLabel]{A} 
 \Vertex[x=2,style={shading=ball}, NoLabel]{B}
 \Vertex[x=0.3, y=2,style={shading=ball}, NoLabel]{C}
  \Vertex[x=2.3, y=2,style={shading=ball}, NoLabel]{D}
 \Vertex[x=3.2, y=3.0,style={shading=ball}, NoLabel]{E}
 \Vertex[x=3.7, y=2,style={shading=ball}, NoLabel]{F}
 \Vertex[x=4.8, y=2,style={shading=ball}, NoLabel]{G}
  \Vertex[x=3.7, style={shading=ball}, NoLabel]{H}
  \Edge(A)(B)   \Edge(A)(C) \Edge(C)(D)   \Edge(B)(D)
  \Edge(D)(E) 
  \Edge(E)(F)   \Edge(F)(G) \Edge (E)(G)
  \end{tikzpicture}  \qquad \\
  \hline  
\qquad  \qquad Disconnected &  Omitted Point \qquad \qquad   \\
  \hline
\end{tabular}  \end{figure}

If $D$ has valence $3$, choose $1$ of the remaining $4$ points and connect it to the $4$-cycle at $D$. We need to select three more lines involving the non-$4$-cycle points, and we must avoid forming  a $3$-cycle which would be a connected component, contradicting the hypothesis, or else leave one point omitted. There are $4$ ways to generate a $3$-cycle among the $4$ remaining points, so there are $\binom{4}{1} \left( \binom{6}{3} -4 \right)=4 \cdot 16=64$ complexes here.
\end{proof}

\begin{lemma}
The number of connected proper line complexes with  a unique $4$-cycle with $2$ vertices of valence greater than $2$ is:
\[
1092.
\]
\end{lemma}
\begin{proof}

Choose $2$ points on the $4$-cycle, $D,B$ and consider their possible valences: $5+3, $4+4$, $, $4+3$, $3+3$, with transposed orderings suppressed. (If $D$ has valence $5$, the  complex has only one additional line besides the $4$-cycle and lines through $D$. This last line must emanate from $B$; hence $5+4$ and other heavy combinations are excluded.)

\begin{figure}[H]
\begin{tikzpicture}
 \Vertex[x=0,style={shading=ball}, NoLabel]{A} 
 \Vertex[x=2,style={shading=ball}, 
 ]{B}
 \Vertex[x=0.3, y=2,style={shading=ball}, NoLabel]{C}
  \Vertex[x=2.3, y=2,style={shading=ball}, 
  ]{D}
 \Vertex[x=3.2, y=3.0,style={shading=ball}, NoLabel]{E}
 \Vertex[x=4,style={shading=ball}, NoLabel]{F}
 \Vertex[x=4.3, y=2,style={shading=ball}, NoLabel]{G}
  \Vertex[x=3.1, y=0.1,style={shading=ball}, NoLabel]{H}
  \Edge(A)(B)   \Edge(A)(C) \Edge(C)(D)   \Edge(B)(D)
  \Edge(D)(E)   \Edge(D)(F)
  \Edge(D)(G) \Edge(B)(H)
  \end{tikzpicture}  
\end{figure}

Suppose $D,B$ have valence $5+3$. We choose one of the $4$ points not on the $4$-cycle to be the vertex connected to $B$ that does not belong to the $4$-cycle, and all else is determined. So there are $2 \cdot 4=8$ complexes of this form, the ``$2$" accounting for transposing $D$ and $B$.

\begin{figure}[H]
\begin{tikzpicture}
 \Vertex[x=0,style={shading=ball}, NoLabel]{A} 
 \Vertex[x=2,style={shading=ball}, 
 ]{B}
 \Vertex[x=0.3, y=2,style={shading=ball}, NoLabel]{C}
  \Vertex[x=2.3, y=2,style={shading=ball}, 
  ]{D}
 \Vertex[x=3.2, y=3.0,style={shading=ball}, NoLabel]{E}
 \Vertex[x=4,style={shading=ball}, NoLabel]{F}
 \Vertex[x=4.3, y=2,style={shading=ball}, NoLabel]{G}
  \Vertex[x=3.1, y=0.1,style={shading=ball}, NoLabel]{H}
  \Edge(A)(B)   \Edge(A)(C) \Edge(C)(D)   \Edge(B)(D)
  \Edge(D)(E)   
  \Edge(B)(F)
  \Edge(D)(G) \Edge(B)(H)
  \end{tikzpicture}  
\end{figure}

Suppose $D,B$ have valence $4+4$. We choose two ``external" points for $D$ and then the external points for $B$  are determined. There are $\binom{4}{2}=6$ cases here. 

\begin{figure}[H]
\begin{tikzpicture}
 \Vertex[x=0,style={shading=ball}, NoLabel]{A} 
 \Vertex[x=2,style={shading=ball}, 
 ]{B}
 \Vertex[x=0.3, y=2,style={shading=ball}, NoLabel]{C}
  \Vertex[x=2.3, y=2,style={shading=ball}, 
  ]{D}
 \Vertex[x=3.2, y=3.0,style={shading=ball}, NoLabel]{E}
 \Vertex[x=4,style={shading=ball}, NoLabel]{F}
 \Vertex[x=4.3, y=2,style={shading=ball}, NoLabel]{G}
  \Vertex[x=3.1, y=0.1,style={shading=ball}, NoLabel]{H}
  \Edge(A)(B)   \Edge(A)(C) \Edge(C)(D)   \Edge(B)(D)
  \Edge(D)(E)   
  \Edge(D)(G) 
  \Edge(B)(H)  \Edge(F)(H)
  \end{tikzpicture}  
\end{figure}

Suppose $D,B$ have valence $4+3$. Choose $2$ of the external points
to yield valence $4$ for $D$,  then $1$ of $2$ possible points to yield valence $3$ for $B$, then connect the remaining point to one of the $3$ just connected points, so that it will not be omitted. We need to double the count to allow for $4+3=3+4$. Hence we have $\binom{4}{2} \cdot 2 \cdot 3 \cdot 2=36 \cdot 2=72$ complexes. 

\begin{figure}[H]
\begin{tikzpicture}
 \Vertex[x=0,style={shading=ball}, NoLabel]{A} 
 \Vertex[x=2,style={shading=ball}, 
 ]{B}
 \Vertex[x=0.3, y=2,style={shading=ball}, NoLabel]{C}
  \Vertex[x=2.3, y=2,style={shading=ball}, 
  ]{D}
 \Vertex[x=3.2, y=3.0,style={shading=ball}, 
 ]{E}
 \Vertex[x=4,style={shading=ball}, NoLabel]{F}
 \Vertex[x=4.3, y=2,style={shading=ball}, NoLabel]{G}
  \Vertex[x=3.1, y=0.1,style={shading=ball}, 
  ]{H}
  \Edge(A)(B)   \Edge(A)(C) \Edge(C)(D)   \Edge(B)(D)
  \Edge(D)(E)   
  \Edge(F)(H) 
  \Edge(B)(H)  \Edge(F)(G)
  \end{tikzpicture}  
\end{figure}

Suppose $D,B$ have valence $3+3$. Choose $1$ of $4$ points, $E$, to make a valence $3$ point in $D$, then $1$ of $3$, $H$, to make valence $3$ for $B$. We have to choose two additional  lines involving $E,H$ and the two remaining points, and $EH$ is forbidden (lest a new $4$-cycle be introduced), so we have to choose $2$ lines from the remaining $6-1=5$. But we cannot have both lines avoiding one the two remaining external points, nor both avoiding the other, so we have $\binom{5}{2}-2=8$. Overall, we have $\binom{4}{2} \cdot \left( \binom {5}{2} -2 \right)=12 \cdot 8=96.$

Adding up, we have
$$
\binom{4}{2} \left[ 8+6+72+96 \right] =6 \cdot 
182 = 1092,
$$
as asserted.
\end{proof}

\begin{lemma}
The number of connected proper line complexes with  a unique $4$-cycle with $3$ vertices of valence greater than $2$ is:
$$
432.
$$
\end{lemma}
\begin{proof}
Choose the $3$ points on the $4$-cycle with valence greater than $2$. The valence distribution must of the form $3+3+4$ or $3+3+3$, as the complex has just $8$ lines.

\begin{figure}[H]
\begin{tikzpicture}
 \Vertex[x=0,style={shading=ball}, NoLabel]{A} 
 \Vertex[x=2,style={shading=ball}, NoLabel]{B}
 \Vertex[x=0.3, y=2,style={shading=ball}, NoLabel]{C}
  \Vertex[x=2.3, y=2,style={shading=ball}, NoLabel]{D}
 \Vertex[x=3.2, y=3.0,style={shading=ball}, NoLabel]{E}
 \Vertex[x=1.1, y=0.1, style={shading=ball}, NoLabel]{F}
 \Vertex[x=4.3, y=2,style={shading=ball}, NoLabel]{G}
  \Vertex[x=3.1, y=0.1,style={shading=ball}, NoLabel]{H}
  \Edge(A)(B)   \Edge(A)(C) \Edge(C)(D)   \Edge(B)(D)
  \Edge(D)(E)   
  \Edge(B)(H)  \Edge(D)(G) \Edge(A)(F)
  \end{tikzpicture}  
\end{figure}

In the $3+3+4$ configuration, we choose $1$ point on the $4$-cycle with valence $2$. Then we choose $1$ of the remaining $3$ points on the $4$-cycle with valence $4$; the remaining $2$ points have valence $3$. Then we choose $2$ of the $4$ off-cycle (external)  points to connect to the valence $4$ vertex. The remaining $2$ off-cycle points can be connected to the $2$ valence $3$ points in $2$ ways. This gives
$$
\binom {4}{1} \binom {3}{1} \binom {4}{2} \cdot 2 = 144 \textrm{  examples}.
$$

\begin{figure}[H]
\begin{tikzpicture}
 \Vertex[x=0,style={shading=ball}, NoLabel]{A} 
 \Vertex[x=2,style={shading=ball}, NoLabel]{B}
 \Vertex[x=0.3, y=2,style={shading=ball}, NoLabel]{C}
  \Vertex[x=2.3, y=2,style={shading=ball}, NoLabel]{D}
 \Vertex[x=3.2, y=3.0,style={shading=ball}, NoLabel]{E}
 \Vertex[x=1.1, y=0.1, style={shading=ball}, NoLabel]{F}
 \Vertex[x=4.3, y=2,style={shading=ball}, NoLabel]{G}
  \Vertex[x=3.1, y=0.1,style={shading=ball}, NoLabel]{H}
  \Edge(A)(B)   \Edge(A)(C) \Edge(C)(D)   \Edge(B)(D)
  \Edge(D)(E)   
  \Edge(B)(H)  
  \Edge(H)(G) 
  \Edge(A)(F)
  \end{tikzpicture}  
\end{figure}

In the $3+3+3$ configuration we choose $1$ of the $4$ points of the $4$-cycle to have valence $2$; the others will have valence $3$. Choose $1$ of the $4$ off-cycle points, which will not be connected to the $4$-cycle. The other $3$ off-cycle points can be connected to the valence $3$ cycle points in $3!$ ways. The off-cycle point not connected to the cycle must be connected to one of the remaining $3$ off-cycle points. This can be done in $3$ ways.
In all, we have have:
\[
\binom {4}{1} \binom{4}{1} (3!) 3 = 288 \textrm{ examples}.
\]
Adding up, $288+144=432$.
\end{proof}

\begin{lemma}
The number of connected proper line complexes with a unique $4$-cycle with $4$ points of valence greater than $2$ is $4!=24.$
\end{lemma}

We hold the truth of this last Lemma to be self-evident, though a picture make come to mind of an upside down table with four legs pointing to the ceiling. (Note that \emph{connected} is superfluous in the statement of the Lemma.)

\section{Admissible Complexes: a complete count}

The grand total is:
$$
1456875- \left[ 210 \cdot  (24+432+1092+500)+5880+315+80640+2520\right]=937440.
$$
Note that the $210$ multiplier accounts for the number of ways to choose a  $4$-cycle among $8$ points for lemmas above assuming a fixed $4$-cycle has been chosen. 
{\small
(There are \( \binom{8}{4}=70 \) ways to choose $4$ points out of $8$, and \( \frac{4!}{4 \cdot 2} =3\) ways to form an unoriented $4$-cycle out of $4$ points; oh, and \( 70 \cdot 3 = 210 \) .)}

%
%


\end{document}